\documentclass[12pt]{article}

\usepackage{amsmath,amssymb,url,amsthm}

\newtheorem{theorem}{Theorem}
\newtheorem{lemma}[theorem]{Lemma}
\newtheorem{problem}{Open Problem}

\newcommand{\codeg}{\mathrm{codeg}}

\begin{document}

\title{The existence of $k$-radius sequences}
\author{Simon R. Blackburn\\
Department of Mathematics\\
Royal Holloway, University of London\\
Egham, Surrey TW20 0EX, United Kingdom\\
\texttt{s.blackburn@rhul.ac.uk}}
\maketitle

\begin{abstract}
Let $n$ and $k$ be positive integers, and let $F$ be an alphabet of
size $n$.  A sequence over $F$ of length $m$ is a \emph{$k$-radius
sequence} if any two distinct elements of $F$ occur within distance
$k$ of each other somewhere in the sequence. These sequences were
introduced by Jaromczyk and Lonc in 2004, in order to produce an
efficient caching strategy when computing certain functions on large
data sets such as medical images.

Let $f_k(n)$ be the length of the shortest $n$-ary $k$-radius
sequence. The paper shows, using a probabilistic argument, that whenever $k$ is fixed and $n\rightarrow\infty$
\[
f_k(n)\sim \frac{1}{k}\binom{n}{2}.
\]
The paper observes that the same argument generalises to the situation when we require the following stronger property for some integer $t$ such that $2\leq t\leq k+1$: any $t$ distinct elements of $F$ must simultaneously occur within a distance $k$ of each other somewhere in the sequence.
\end{abstract}

\section{Introduction}
\label{sec:introduction}

Let $n$ and $k$ be positive integers, and let $F$ be an alphabet of
size $n$.  A sequence $a_1,a_2,\ldots,a_{m}$ over $F$ of length
$m$ is a \emph{$k$-radius sequence} if for all $x,y\in F$ there
exists $i,j\in\{1,2,\ldots,m\}$ such that $a_i=x$, $a_j=y$ and
$|i-j|\leq k$. The following is an example of an $8$-ary
$3$-radius sequence over the alphabet $F=\{0,1,2,3,4,5,6,7\}$:
\[
0,1,2,3,4,5,6,7,0,1,2,4,5,6,3,7.
\]

We write $f_k(n)$ for the length of the shortest $n$-ary $k$-radius sequence;
so the example above shows that $f_3(8)\leq 16$.

The concept of a $k$-radius sequence was introduced by Jaromczyk and
Lonc~\cite{JL}. They were interested in these sequences so they could
design an efficient caching strategy to compute a function that
depends on comparing pairs of a sequence of large data sets such as
medical images (see the discussion in Section~\ref{sec:comments} below).

Ghosh~\cite{Ghosh} showed that
\[
f_1(n)=\begin{cases} \binom{n}{2}+1&\text{when $n$ is odd;}\\
\binom{n}{2}+n/2&\text{when $n$ is even.}
\end{cases}
\]
Jaromczyk and Lonc~\cite{JL} showed that
$f_2(n)=\frac{1}{2}\binom{n}{2}+O(n^2/\log n)$, and gave a
construction for $k$-radius sequences of the right order of magnitude
but with a leading term that is not tight. Chee, Ling, Tan and
Zhang~\cite{CheeLing} provided good constructions for $n$-ary
$2$-radius sequences for small values of $n$. Blackburn and
McKee~\cite{BlackburnMcKee} showed how to construct asymptotically
good $k$-radius sequences for many values of $k$. In particular, their
constructions show that $f_{k}(n)=\frac{1}{k}\binom{n}{2}+O(n^2/\log
n)$ whenever $k\leq 194$, or $k+1$ is a prime, or $2k+1$ is a
prime. They asked whether $\lim_{n\rightarrow\infty}
f_k(n)/\binom{n}{2}$ exists and is equal to $1/k$. The main purpose of
this paper is to answer this question positively, by proving the
following theorem.

\begin{theorem}
\label{thm:main}
Let $k$ be a fixed positive integer. Then
\[
f_k(n)\sim \frac{1}{k}\binom{n}{2},
\]
as $n\rightarrow\infty$.
\end{theorem}

We use probabilistic methods, our main tool being Pippenger and Spencer's version of the Frankl-Rodl theorem on the size of the matchings in a quasi-random hypergraph~\cite{FranklRodl,PippengerSpencer} (see Section~\ref{sec:comments}).

Theorem~\ref{thm:main} is proved in the next section. In the final section of the paper we observe that the proof of this theorem can be generalised (Theorem~\ref{thm:general}) to the situation where we are interested in subsets of $t$ elements, rather than pairs of elements, from the alphabet. The final section also contains various comments and open problems.

\section{The proof of Theorem~\ref{thm:main}}
\label{sec:proof}

We begin with an elementary lemma which establishes a lower bound on
$f_k(n)$. Jaromczyk and Lonc~\cite{JL} prove a slightly stronger
version of this lemma in their paper.

\begin{lemma}
\label{lem:lower_bound}
For all positive integers $n$ and $k$,
\[
f_k(n)>\frac{1}{k}\binom{n}{2}.
\]
\end{lemma}
\begin{proof}
Let $a_1,a_2,\ldots,a_{m}$ be an $n$-ary $k$-radius
sequence. There are less than $km$ pairs of the form
$\{a_i,a_{i+z}\}$ where $i,i+z\in\{1,2,\ldots,m\}$ and
$1\leq z\leq k$. The $k$-radius sequence property implies that
every unordered pair of alphabet symbols must occur at least once as a pair
$\{a_i,a_{i+z}\}$ for some $i$ and $z$, and so $km>
\binom{n}{2}$.
\end{proof}

To provide an upper bound on $f_k(n)$, we use a well-known theorem in
hypergraph theory. Recall that a hypergraph $\Gamma$ is
\emph{$r$-uniform} if all its hyperedges have cardinality $r$. The
\emph{degree} $\deg(v)$ of a vertex $v\in\Gamma$ is the number of
hyperedges containing $v$; the \emph{codegree} $\codeg(v,w)$ of a pair of
distinct vertices $v,w\in\Gamma$ is the number of hyperedges containing
both $v$ and $w$. A \emph{cover} is a set of hyperedges in $\Gamma$
whose union is equal to the set of all vertices of $\Gamma$.

\begin{theorem}
\label{thm:hypergraph}
Fix an integer $r$ and a positive real number $\delta$. Then there
exists an integer $n_0$ and a positive real number $\delta'$ with
the following property.

Let $\Gamma$ be an $r$-uniform hypergraph on $n$ vertices, where
$n\geq n_0$. Suppose that all vertices of $\Gamma$ have degree $d$ for
some integer $d$. Let $c=\max \codeg(u,v)$, where the maximum is taken
over all pairs of distinct vertices $u,v\in \Gamma$. If $c\leq
\delta' d$, then there exists a cover consisting of at most
$(1+\delta)n/r$ hyperedges.
\end{theorem}

Theorem~\ref{thm:hypergraph} can be proved using a second-moment technique that Alon and Spencer~\cite{AlonSpencer} call the `R\"odl
nibble' (see~\cite{Rodl,FranklRodl} or~\cite[Theorem~8.4]{Furedi} for example). Also see Pippenger and
Spencer~\cite[Theorem~1.1]{PippengerSpencer} for a stronger
result.

\begin{proof}[Proof of Theorem~\ref{thm:main}]
We need to prove that $\lim_{n\rightarrow\infty} f_k(n)/\binom{n}{2} =
1/k$. Now $f_k(n)/\binom{n}{2}>1/k$ by
Lemma~\ref{lem:lower_bound}. Let $\epsilon$ be a fixed positive real
number. To prove the theorem, it suffices to show that for all
sufficiently large integers $n$, we have that $f_k(n)/\binom{n}{2}\leq
(1+\epsilon)/k$.

Choose an integer $\ell$ and a positive real number $\delta$ such that
$\ell\geq k$ and
\[
\frac{(1+\delta)}{1-\frac{1}{2}(k+1)/\ell}<(1+\epsilon).
\]

Let $n$ be an integer such that $n\geq \ell$, and let $F$ be a set of
cardinality $n$. Define a hypergraph $\Gamma_n$ as follows. The
vertices of $\Gamma_n$ are the $\binom{n}{2}$ unordered pairs $\{x,y\}$
where $x,y\in F$. The hyperedges of $\Gamma_n$ are the $n(n-1)\cdots
(n-(\ell-1))$ sequences $\mathbf{b}=b_1,b_2,\ldots,b_{\ell}$ of length $\ell$
over $F$ whose entries $b_i$ are all distinct. We define a vertex
$\{x,y\}$ to lie in a hyperedge $\mathbf{b}$ whenever $x$ and $y$ occur in
$\mathbf{b}$ within a distance of $k$; more precisely, whenever there
exist $i,j\in\{1,2,\ldots ,\ell\}$ such that $b_i=x$, $b_j=y$ and
$|i-j|\leq k$.

Let $r$ be the number of ways of choosing an (unordered) pair of distinct positions in a sequence of length $\ell$, where the positions are at most a distance $k$ apart. So
\[
r=(\ell-k)k+\sum_{i=0}^{k-1} i = \ell k - \tfrac{1}{2}k(k+1).
\]
Clearly $r$ does not depend on $n$.
Since the entries of $\mathbf{b}$ are distinct, every hyperedge in
$\Gamma_n$ contains exactly $r$ vertices,
and so $\Gamma_n$ is an $r$-uniform hypergraph. The degree $d$ of any vertex $v\in \Gamma_n$ is equal
to $2r (n-2)(n-3)\cdots(n-(\ell-1))$, which is of the order of
$n^{\ell-2}$. The codegree of distinct vertices $v,w\in\Gamma_n$
depends on whether $v$ and $w$ are intersecting when thought of as
pairs of elements of $F$. But in either case
$\codeg(v,w)=O(n^{\ell-3})=o(d)$. So Theorem~\ref{thm:hypergraph}
implies that for all sufficiently large integers $n$ there exists a
cover $\mathbf{b}_1,\mathbf{b}_2,\ldots,\mathbf{b}_s$ for $\Gamma_n$
consisting of $s$ hyperedges, where $s\leq (1+\delta)\binom{n}{2}/r$.

The definition of $\Gamma_n$ and the fact that the sequences
$\mathbf{b}_i$ form a cover show that the concatenation of
$\mathbf{b}_1,\mathbf{b}_2,\ldots,\mathbf{b}_s$ is a $k$-radius
sequence. The length of this sequence is $\ell s$, and
\begin{align*}
\ell s&\leq \ell(1+\delta)\binom{n}{2}/r\\
&=\frac{1}{k}\binom{n}{2}\frac{\ell(1+\delta)}{\ell-\frac{1}{2}(k+1)}\\
&=\frac{1}{k}\binom{n}{2}\frac{(1+\delta)}{1-\frac{1}{2}(k+1)/\ell}\\
&<\frac{1}{k}\binom{n}{2}(1+\epsilon).
\end{align*}
So $f_k(n)/\binom{n}{2}\leq (1+\epsilon)/k$ for all
sufficiently large integers $n$, as required.
\end{proof}

\section{Comments}
\label{sec:comments}

We have found the leading term for $f_k(n)$ as $n\rightarrow\infty$
with $k$ fixed using probabilistic methods. It would be very
interesting to search for explicit constructions of $k$-radius
sequences that are asymptotically good for any value of $k$. (The
constructions of Jaromczyk and Lonc~\cite{JL} and of Blackburn and
McKee~\cite{BlackburnMcKee} only lead to asymptotically good
constructions for some values of $k$.) The following problem would also be very interesting:

\begin{problem}
\label{prob:k_radius}
Provide an upper bound (using explicit or probabilistic
methods) of the form
\[
f_k(n)\leq \frac{1}{k}\binom{n}{2}+g(n),
\]
where $g(n)$ is a function of $n$ that grows significantly more slowly
than $n^2$.
\end{problem}

\emph{Note added in final revision:} A recent preprint of Jaromzcyk, Lonc and Truszczynski~\cite{JLT} provides some beautiful recursive constructions of $k$-radius sequences, solving Open Problem~\ref{prob:k_radius}. Indeed, they show that we may take $g(n)=O(n^{1+\epsilon})$ for any positive real number $\epsilon$. They also give optimal constructions of $2$-radius sequences when $n=2p$ with $p$ a prime.

We now discuss the caching application that motivated Jaromczyk and
Lonc in a little more detail. Suppose we have a total of $n$ medical images,
and we wish to compute some function which depends on all pairs of these
images. We assume that the computation involving each pair of images
is computationally intensive, so we wish to place these images in our
cache before carrying out this computation. We assume our cache can
hold up to $k+1$ images at one time. Then an $n$-ary $k$-radius
sequence will enable us to design an efficient caching strategy, as
follows. Let $a_1,a_2,\ldots,a_{m}$ be an $n$-ary $k$-radius
sequence. Suppose we load image $a_t$ into our cache at time $t$,
using a first-in first-out caching strategy. So at time $t$ (for
$t\geq k+1$) our cache holds the images $a_{t-k},a_{t-k+1},\ldots
,a_t$. The property of being a $k$-radius sequence implies that any
pair of alphabet symbols occurs in some window of length $k+1$ in the
sequence, and so any pair of images simultaneously lies in our cache
at some point. Short sequences correspond to efficient caching
strategies for this problem.

We might ask what the consequences are of
removing our insistence on a first-in first-out strategy in the application above. But whatever
caching strategy is used it is clear that at
most $k$ new pairs of images are introduced into our cache at every
time period: the bound of Lemma~\ref{lem:lower_bound} holds for any
caching strategy. So the results of this paper show that imposing the
restriction to a first-in first-out strategy does not
affect the asymptotic efficiency. We should also remark that when we are not imposing the restriction to a first-in first-out strategy there is a simple caching method that gives asymptotically tight results, which can be described as follows. We begin by loading the first batch of $k$ images $1,2,\ldots,k$ into the cache. Our cache can store one more image: keeping our initial batch of images in our cache, we load all the remaining images in turn. So at time $k+i$ where $1\leq i\leq n-k$, the cache holds images $1,2,3,\ldots,k$ and $k+i$. We then continue with the next batch of $k$ images $k+1,k+2,\ldots,2k$: at time $n+k+i$ where $1\leq i\leq n-2k$ the cache holds images $k+1,k+2,\ldots,2k$ and $2k+i$. We continue in this way, first loading a batch of $k$ images into our cache and then using the remaining space to load each of the later images in turn.

The results of this paper are easily generalised to a wider class of
combinatorial objects. Let $k$, $t$ and $n$ be fixed positive
integers, with $t\leq n$. Let $F$ be an alphabet of cardinality
$n$. We may define a \emph{$t$-subset $k$-radius sequence over $F$} to
be a finite sequence $a_1,a_2,\ldots ,a_m$ over $F$ such that for all
$t$-subsets $X\subseteq F$, there exists $i\in\{1,2,\ldots,m-k\}$ such
that
\[
X\subseteq\{a_i,a_{i+1},\ldots,a_{i+k}\}.
\]
So a $k$-radius sequence satisfies this definition in the special case
when $t=2$. Let $f_{t,k}(n)$ be the length of the shortest $t$-subset
$k$-radius sequence.
\begin{theorem}
\label{thm:general}
Let $k$ and $t$ be fixed integers such that $2\leq t\leq k+1$. Then
\[
f_{t,k}(n)\sim \frac{1}{\binom{k}{t-1}}\binom{n}{t}
\]
as $n\rightarrow\infty$.
\end{theorem}
\begin{proof}
The lower bound follows by observing that each new element added to a sequence can `cover' at most $\binom{k}{t-1}$ new subsets $X$ (as these new subsets must involve the new element).

The upper bound follows as in the proof of Theorem~\ref{thm:main}. So we define the hypergraph $\Gamma_n$ to have hyperedges as before, with vertices the $t$-subsets of the $n$-set $F$, and with a vertex lying in a hyperedge $\mathbf{b}$ if and only if the subset is contained in a set of $k+1$ consecutive elements of the sequence $\mathbf{b}$. The graph $\Gamma_n$ is $r$-uniform, where $r=\ell\binom{k}{t-1}+h(k,t)$ for some fixed function $h$ of $k$ and $t$. The degree of a vertex in $\Gamma_n$ does not depend on the vertex and is of the order of $n^{\ell-t}$, whereas the codegree of a pair of vertices depends on the size of intersection of the subsets the vertices are identified with, but is at most $O(n^{\ell-t-1})$. We may use Theorem~\ref{thm:hypergraph} to obtain a small cover for $\Gamma_n$, and then concatenate the resulting sequences in this cover to obtain a short $t$-subset $k$-radius sequence, just as in the proof of Theorem~\ref{thm:main}.
\end{proof}

\begin{problem}
Find good explicit constructions of $t$-subset $k$-radius sequences.
\end{problem}

This problem has been considered in the case $t=k+1$ by Lonc, Traczyk, and Truszczynski~\cite{LTT}. The authors show that $f_{k+1,k}(n)=\binom{n}{k}+O(n^{\lfloor k/2\rfloor})$, determine $f_{3,2}(n)$ exactly and determine $f_{4,3}(n)$ and $f_{6,5}(n)$ for infinitely many values of $n$.

The corresponding packing rather than covering problem is also
interesting combinatorially (although we do not know of an
application). Here we may define a \emph{packing $t$-subset $k$-radius
sequence over $F$} to be a sequence $a_1,a_2,\ldots,a_m$ over $F$ with
the property that any $t$-subset $X\subseteq F$ only occurs as a
subset of $\{a_i,a_{i+1},\ldots,a_{i+k}\}$ in at most one position in
the sequence. More precisely, we require that for all $t$-subsets
$X\subseteq F$ there exists at most one choice for an increasing sequence $z_1,z_2,\ldots,z_t$ of integers
such that
\[
X=\{a_{z_1},a_{z_2},\ldots,a_{z_t}\}
\]
and where $|z_t-z_1|\leq k$.

\begin{problem} Define $F_{t,k}(n)$ to be
the length of the longest packing $t$-subset $k$-radius sequence. Find
good asymptotic lower bounds on $F_{t,k}(n)$, either using probabilistic or
explicit constructions.
\end{problem}

This problem has been considered in the case when $t=k+1$ by Curtis, Hines, Hurlbert and Moyer~\cite{CHHM} under the name of Ucycle packings. The authors prove that for any $t$, we have
\[
F_{t,t-1}(n)=(1-o(1))\binom{n}{t}.
\]
Their work was motivated by the concept of a universal cycle; see~\cite{CDG,Hurlbert,Jackson}.

\paragraph{Acknowledgement} The author would like to thank Jason Crampton for pointing out the simple caching construction described in Section~\ref{sec:comments}, and the referees for suggesting several improvements in the exposition and bibliography.

\end{document}